\documentclass[12pt,a4paper,fleqn]{article}
\usepackage{a4wide,amsfonts,amsmath,latexsym,amssymb,euscript,graphicx,units,mathrsfs}

\usepackage[english]{babel}
\usepackage{graphicx}
\usepackage{color}
\usepackage{amssymb}
\usepackage{amssymb}
\usepackage[T1]{fontenc}
\usepackage{latexsym}
\usepackage{xypic}
\usepackage{eufrak}
\usepackage{euscript}
\usepackage{amsfonts,amsmath}
\usepackage{verbatim}
\usepackage{fancyhdr}
\usepackage{mathrsfs}
\usepackage{units}

\title{An It\^{o}'s type formula
for the fractional Brownian motion in Brownian time}

\author{Ivan Nourdin\footnote{Institut Elie Cartan, UMR 7502, Nancy Universit\'e - CNRS - INRIA; {\tt inourdin@gmail.com};
IN is supported in part by the (french) ANR grant `Malliavin, Stein and Stochastic Equations with Irregular Coefficients' [ANR-10-BLAN-0121]
} \ and Raghid Zeineddine\footnote{Institut Elie Cartan, UMR 7502, Nancy Universit\'e - CNRS - INRIA; {\tt raghid.zeineddine@univ-lorraine.fr}}}

\newtheorem{prop}{Proposition}[section]

\newtheorem{lemme}[prop]{Lemma}
\newtheorem{lemma}[prop]{Lemma}

\newtheorem{thm}[prop]{Theorem}

\renewcommand{\geq}{\geqslant}
\def\leq{\leqslant}

\newcommand{\N}{\mathbb{N}}
\newcommand{\Z}{\mathbb{Z}}

\newcommand{\R}{\mathbb{R}}

\def\1{{\mathbf{1}}}

\def\1{{\mathbf{1}}}
\def\0.5{{\frac{1}{2}}}

\newenvironment{proof}[1]{\begin{trivlist}\item {\it
\bf Proof.}\quad} {\qed\end{trivlist}}

\newcommand{\qed}{\nopagebreak\hspace*{\fill}
{\vrule width6pt height6ptdepth0pt}\par}

\begin{document}
\maketitle
\begin{abstract}
Let $X$ be a (two-sided) fractional Brownian motion of Hurst parameter $H\in (0,1)$ and let
$Y$ be a standard Brownian motion independent of $X$. Fractional Brownian motion in Brownian motion time (of index $H$), recently studied in \cite{13}, is by definition the process $Z=X\circ Y$. It is a continuous, non-Gaussian process with stationary increments, which is selfsimilar of index $H/2$.
The main result of the present paper is an It\^{o}'s type formula for $f(Z_t)$, when $f:\R\to\R$ is smooth and $H\in [ 1/6,1)$.
When $H>1/6$, the change-of-variable formula we obtain is similar to that of the classical calculus.
In the critical case $H=1/6$, our change-of-variable formula is in law and involves the third derivative of $f$ as well as an extra Brownian motion independent of the pair $(X,Y)$. We also discuss briefly the case $H<1/6$.
\end{abstract}
 \textbf{Keywords:} Fractional Brownian motion in Brownian time; change-of-variable formula in law; Malliavin calculus.

\section {Introduction}

If $f:\R_+\to\R$ is $C^1$ then
$f(t)=f(0)+\int_0^t f'(s)ds$ for all $t\geq 0$ whereas, if $W$ is a standard Brownian motion and if $f:\R\to \R$ is $C^2$ then, by the It\^o's formula,
\begin{equation}\label{ito}
f(W_t)=f(0)+\int_0^t f'(W_s)d^-W_s + \frac{1}{2}\int_0^t f''(W_s)ds,\quad t\geq 0.
\end{equation}
In (\ref{ito}) the It\^o integral, namely
\begin{equation}\label{forward}
\int_0^t X_sd^-Y_s := \lim_{n\to\infty}\sum_{k=0}^{\lfloor 2^n t\rfloor -1}
X_{k2^{-n}}(Y_{(k+1)2^{-n}}-Y_{k2^{-n}}),
\end{equation}
is of forward type.
It is well-known that the additional bracket term $\frac{1}{2}\int_0^t f''(W_s)ds$ appearing in (\ref{ito}) comes from
the non-negligibility of the quadratic variation of $W$ in the large limit; more precisely,
\begin{equation}\label{quadratic}
\sum_{k=0}^{\lfloor 2^n t\rfloor -1} (W_{(k+1)2^{-n}}-W_{k2^{-n}})^2
\,\,\,
\overset{\rm a.s.}{\longrightarrow}\,\,\,t\quad\mbox{as $n\to\infty$}.
\end{equation}

Introducing a family $\{B^H\}_{H\in(0,1)}$ of fractional Brownian motions parametrized by the Hurst parameter $H$ may help to
reinterpret (\ref{ito}) in a more dynamical way. Let us elaborate this point of view further. Recall that $B^{\frac12}$ is nothing but the standard Brownian motion, whereas
$B^1$ is the process $B^1_t=tG$, $t\geq 0$, $G\sim N(0,1)$. The extension
of (\ref{quadratic}) to any $H\in(0,1)$ is well-known: one has
\begin{equation}\label{quadratic-fbm}
2^{n(2H-1)}\sum_{k=0}^{\lfloor 2^n t\rfloor -1} (B^H_{(k+1)2^{-n}}-B^H_{k2^{-n}})^2\,\,\,
\overset{\rm a.s.}{\longrightarrow}\,\,\,t\quad\mbox{as $n\to\infty$}.
\end{equation}
Based on (\ref{quadratic-fbm}), it is then not difficult to prove the following two facts:
\begin{enumerate}
\item If $H>\frac12$ and $f:\R\to\R$ is $C^2$ (actually, $C^1$ is enough), then $\int_0^t f'(B^H_s)d^-B^H_s$ exists
as a limit in probability and we have
\[
f(B^H_t)=f(0)+\int_0^t f'(B^H_s)d^-B^H_s,\quad t\geq 0.
\]
\item If $H<\frac12$, then
\[
\int_0^t B^H_s d^-B^H_s = -\infty\quad\mbox{a.s.},
\]
meaning that there is no possible change-of-variable formula for
$f(x)=x^2$.
\end{enumerate}
Thus, $H=\frac12$ appears to be a critical value for the change-of-variable formula involving the forward integral (\ref{forward}). This is because it is precisely the value from which the sign of $2H-1$ changes in (\ref{quadratic-fbm}).
The chain rule being (\ref{ito}) in the critical case $H=\frac12$, one has a complete picture for the forward integral (\ref{forward}).\\

To go one step further, one may wonder what kind of change-of-variable formula one would obtain after replacing the definition (\ref{forward}) by
its symmetric counterpart, namely
\begin{equation}\label{symmetric}
\int_0^t X_sd^\circ Y_s := \lim_{n\to\infty} \sum_{k=0}^{\lfloor 2^n t\rfloor -1}
\frac12\big(X_{k2^{-n}}+X_{(k+1)2^{-n}}\big)(Y_{(k+1)2^{-n}}-Y_{k2^{-n}})
\end{equation}
(provided the limit exists in some sense).  As it turns out, it is arguably a much more difficult problem, which has been solved only recently.
In this context, the crucial quantity is now the cubic variation. And this latter is known to
satisfy, for any $H<\frac12$,
\begin{equation}\label{cubic-fbm}
2^{n(3H-\frac12)}\sum_{k=0}^{\lfloor 2^n t\rfloor -1} (B^H_{(k+1)2^{-n}}-B^H_{k2^{-n}})^3\,\,\,
\overset{\rm law}{\to}\,\,\,N(0,\sigma^2_H)\quad\mbox{as $n\to\infty$}.
\end{equation}
With a lot of efforts, one can prove (see \cite{CN,10} when $H\neq \frac16$ and \cite{11} when $H=\frac16$) the following three facts, which hold for any smooth enough real function $f:\R\to\R$:
\begin{enumerate}
\item If $H>\frac16$ then
$\int_0^t f'(B^H_s)d^\circ B^H_s$ exists
as a limit in probability and one has
\begin{equation}\label{hgrand}
f(B^H_t)=f(0)+\int_0^t f'(B^H_s)d^\circ B^H_s,\quad t\geq 0.
\end{equation}
\item If $H=\frac16$ then $\int_0^t f'(B^\frac16_s)d^\circ B^\frac16_s$ exists
as a stable limit in law and one has, with $W$ a standard Brownian motion independent of $B^\frac16$ and $\kappa_3\simeq 2.322$,
\begin{equation}\label{h16}
f(B^\frac16_t)=f(0)+\int_0^t f'(B^\frac16_s)d^\circ B^\frac16_s-\frac{\kappa_3}{12}\int_0^t f'''(B^\frac16)dW_s,\quad t\geq 0.
\end{equation}
\item If $H<\frac16$ then
\begin{equation}\label{hpetit}
\int_0^t (B^H_s)^2 d^\circ B^H_s \mbox{ does not exist in law}.
\end{equation}
\end{enumerate}
Thus, as we see, the critical value for the symmetric integral is now $H=\frac16$; it is exactly the value of $H$ from which the sign of $3H-\frac12$ changes in (\ref{cubic-fbm}).\\

In \cite{1,2} (see also \cite{3}), Burdzy has introduced the so-called iterated Brownian motion.
This process, which can be regarded as the realization of a Brownian motion on a
random fractal, is defined as
\[
Z_t=X(Y_t), \quad t\geq 0,
\]
where $X$ is a two-sided Brownian motion and $Y$ is a standard (one-sided) Brownian
motion independent of $X$.
Note that $Z$ is self-similar of order $\frac14$ and has stationary increments; hence, in some sense, $Z$ is close to the fractional Brownian motion $B^\frac14$ of index $H=\frac14$. As is the case for $B^\frac14$,
$Z$ is neither a Dirichlet process nor a semimartingale or a Markov process in its own filtration. A crucial question is therefore how to define a stochastic calculus with respect to it. This issue has been tackled by Khoshnevisan and Lewis in \cite{5,6}, where the authors develop a Stratonovich-type stochastic calculus with respect to $Z$, by extensively using techniques based on the properties of some special arrays of Brownian stopping times, as well as on excursion-theoretic arguments. See also the paper \cite{NPejp} which may be seen as a follow-up of \cite{5}.
The formula obtained in \cite{5,6} reads, unsurprisingly (due to (\ref{hgrand}) and the similarities between $Z$ and $B^\frac14$) and losely speaking, as follows:
\begin{equation}\label{KL}
f(Z_t)=f(0)+\int_0^t f(Z_s)d^\circ Z_s,\quad t\geq 0.
\end{equation}

The change-of-variable formula (\ref{KL})  is of the same kind  than (\ref{hgrand}). In view of what has been done so far for the fractional Brownian motion $B^H$, aiming to provide an answer to the following problem is somehow natural: can we also reinterpret (\ref{KL}) in a dynamical way, in the spirit of (\ref{hgrand}), (\ref{h16}) and (\ref{hpetit})? To this end, we first need to introduce
a family of processes  that contains the iterated Brownian motion $Z$ as a particular element. The family consisting in the so-called fractional Brownian motions in Brownian time, studied  in \cite{13} by the second-named author, does the job. More specifically, it is the family
$\{Z^H\}_{H\in (0,1)}$ defined as follows:
\[
Z^H_t=X^H(Y_t), \quad t\geq 0,
\]
where $X^H$ is a two-sided fractional Brownian motion of index $H$ and $Y$ is a standard (one-sided) Brownian
motion independent of $X$.
Roughly speaking, in the present paper we are going to show the following three assertions (see Theorem \ref{main} for a precise statement):
for any smooth real function $f:\R\to\R$,
\begin{enumerate}
\item If $H>\frac16$ then
\[
f(Z^H_t)=f(0)+\int_0^t f'(Z^H_s)d^\circ Z^H_s,\quad t\geq 0.
\]
\item If $H=\frac16$ then, with $W$ a standard Brownian motion independent of the pair $(X^\frac16,Y)$ and $\kappa_3\simeq 2.322$,
\begin{equation}\label{h=16}
f(Z^\frac16_t)=f(0)+\int_0^t f'(Z^\frac16_s)d^\circ Z^\frac16_s-\frac{\kappa_3}{12}\int_0^{Y_t} f'''(X^\frac16_s)dW_s,\quad t\geq 0.
\end{equation}
\item If $H<\frac16$, then
\[
\int_0^t (Z^H_s)^2 d^\circ Z^H_s \mbox{ does not exist}.
\]
\end{enumerate}

The formula (\ref{h=16}) is related to a recent line of research in which, by means of Malliavin calculus, one aims to exhibit change-of-variable formulas in law with a correction term which is an It\^o integral with respect to martingale independent of the underlying Gaussian processes. Papers dealing with this problem and which are prior to our work include
\cite{BS,HN1,12,HN3,nourdinjfa,NR,11}; however, it is worthwhile noting that all these mentioned references only deal with Gaussian processes, not with iterated processes (which are arguably more difficult to handle).\\

   A brief outline of the paper is as follows.
   In Section 2, we introduce the framework in which our study takes place and we provide an exact statement of our result, namely Theorem \ref{main}.
   Finally, Section 3 contains the proof of Theorem \ref{main}, which is divided into several steps.

\section{Framework and exact statement of our results}\label{framework}
For simplicity, throughout the paper we remove the superscript $H$, that is, we write $Z$ (resp. $X$) instead of $Z^H$ (resp. $X^H$).

Let $Z$ be a fractional Brownian motion in Brownian time of Hurst parameter $H\in(0,1)$,
defined as
\begin{equation}\label{fbmbt}
Z_t=X(Y_t), \quad t\geq 0,
\end{equation}
where $X$ is a two-sided fractional Brownian motion of parameter $H$ and $Y$ is a standard (one-sided) Brownian
motion independent of $X$.

The paths of $Z$ being very irregular (precisely: H\"older continuous of order $\alpha$ if and only if $\alpha$ is strictly less than $H/2$), we will  not be able to define a stochastic integral with respect to it as the limit of Riemann sums with respect to a {\it deterministic} partition of the time axis. However, a winning idea borrowed from Khoshnevisan and Lewis \cite{5,6} is to approach deterministic partitions by means of random partitions defined in terms of hitting times of the underlying Brownian motion $Y$. As such, one can bypass the random ``time-deformation'' forced by (\ref{fbmbt}), and perform asymptotic procedures by separating the roles of $X$ and $Y$ in the overall definition of $Z$.

Following Khoshnevisan and Lewis \cite{5,6}, we start by introducing the so-called intrinsic skeletal structure of $Z^H$. This structure is defined through a sequence of collections of stopping times (with
respect to the natural filtration of $Y$), noted
\begin{equation}
\mathscr{T}_n=\{T_{k,n}: k\geq0\}, \quad n\geq1, \label{TN}
\end{equation}
which are in turn expressed in terms of the subsequent hitting
times of a dyadic grid cast on the real axis. More precisely, let
$\mathscr{D}_n= \{j2^{-n/2}:\,j\in\Z\}$, $n\geq 1$, be the dyadic
partition (of $\R$) of order $n/2$. For every $n\geq 1$, the
stopping times $T_{k,n}$, appearing in (\ref{TN}), are given by
the following recursive definition: $T_{0,n}= 0$, and
\[
T_{k,n}= \inf\big\{s>T_{k-1,n}:\quad
Y(s)\in\mathscr{D}_n\setminus\{Y(T_{k-1,n})\}\big\},\quad k\geq 1.
\]
Note that the definition of $T_{k,n}$, and
therefore of $\mathscr{T}_n$, only involves the one-sided Brownian
motion $Y$, and that, for every $n\geq1$, the discrete stochastic
process
\[
\mathscr{Y}_n=\{Y(T_{k,n}):k\geq0\}
\]
defines a simple random
walk over $\mathscr{D}_n$.  As shown in
\cite[Lemma 2.2]{5}, as $n$ tends to
infinity the collection $\{T_{k,n}:\,1\leq k \leq 2^nt\}$ approximates the
common dyadic partition $\{k2^{-n}:\,1\leq k \leq 2^nt\}$ of order $n$ of the time interval $[0,t]$. More precisely,
\begin{equation}\label{lemma2.2}
\sup_{0\leq s\leq t} \big| T_{\lfloor 2^n s\rfloor,n}-s\big|\to 0\quad\mbox{almost surely and in $L^2(\Omega)$.}
\end{equation}
Based on this fact, one may introduce the counterpart of (\ref{symmetric}) based on $\mathscr{T}_n$, namely,
\begin{equation}
V_n(f, t)= \sum_{k=0}^{\lfloor 2^n t \rfloor -1} \frac{1}{2}\big(f(Z_{T_{k,n}})+f(Z_{T_{k+1,n}})\big)(Z_{T_{k+1,n}}-Z_{T_{k,n}}).\label{v}
\end{equation}
Let $C_b^{\infty}$ denote the class of those functions  $f:\R\to\R$ that are $C^{\infty}$ and bounded together with their derivatives. We then have the following result.
\begin{thm}\label{main} Let $f \in C_b^{\infty}$ and $t>0$.
\begin{enumerate}
\item If $ H>\frac16$ then
\begin{eqnarray}
 f(Z_t) - f(0) = \int_0^{t}f'(Z_s)d^\circ Z_s,\label{Ito1}
\end{eqnarray}
where $\int_0^{t}f'(Z_s)d^\circ Z_s$ is the limit in probability of $V_n(f', t)$ defined in (\ref{v}) as $n\to\infty$.
\item If $H =\frac16$ then, with $\kappa_3\simeq 2.322$,
\begin{eqnarray}
 f(Z_t) - f(0) + \frac{\kappa_3}{12}\int_0^{Y_{t}}f'''(X_s)dW_s  &\overset{\rm (law)}{=}& \int_0^{t}f'(Z_s)d^{\circ}Z_s. \label{Ito2}
\end{eqnarray}
Here, $\int_0^{t}f'(Z_s)d^{\circ}Z_s$ denotes the limit in law of $V_n(f', t)$ defined in (\ref{v}) as $n\to\infty$
(its existence is part of the conclusion), $W$ is a two-sided Brownian motion independent of the pair $(X,Y)$ defining $Z$,
and the integral with respect to $W$ is understood in the Wiener-It\^o sense.
\item If $H<\frac16$ then
\begin{equation}\label{Ito3}
V_n(x\mapsto x^2,t)\mbox{ does not
converge, even stably in law.}
\end{equation}
This means that there is no way to get a change-of-variable formula for $f(x)=x^3$.
\end{enumerate}
\end{thm}

  \section{ Proof of Theorem \ref{main}}

  \subsection{Elements of Malliavin calculus}
  In this section, we gather some elements of Malliavin calculus we shall need thoughout the proof of Theorem \ref{main}. The reader already familiar with this topic may skip this section.

We continue to denote by  $X = (X_{t})_{t \in \R}$ a two-sided fractional Brownian motion with Hurst parameter $ H \in (0,1).$ That is, $X$ is a zero mean Gaussian process, defined on a complete probability space $(\Omega, \mathscr{A}, P)$, with the covariance function \[ C_{H}(t,s) = E(X_{t}X_{s})=\frac{1}{2}(|s|^{2H} + |t|^{2H} -|t-s|^{2H}),\text{\:\:\:} s,t \in \R.\]
   We suppose that $\mathscr{A}$ is the $\sigma$-field generated by $X$. For all $n \in \N^*$, we let $\mathscr{E}_n$ be the set of step functions on $[-n,n]$, and $\displaystyle{\mathscr{E}:= \cup_n \mathscr{E}_n}$. Set $\xi_t = \textbf{1}_{[0,t]}$ (resp. $\textbf{1}_{[t,0]}$) if $t \geq 0$ (resp. $t < 0$). Let $\mathscr{H}$ be the Hilbert space defined as the closure of $\mathscr{E}$ with respect to the inner product
   \[
    \langle \xi_t, \xi_s \rangle_{\mathscr{H}} = C_{H}(t,s),\quad s,t \in \R.
\]
     The mapping $\xi_t \mapsto X_{t}$ can be extended to an isometry between $\mathscr{H}$ and the Gaussian space $\mathbb{H}_{1}$ associated with $X$. We will denote this isometry by $\varphi \mapsto X(\varphi).$

     Let $\mathscr{F}$ be the set of all smooth cylindrical random variables, i.e. of the form
   \[ F = \phi (X_{t_{1}},...,X_{t_{l}}),\] where $l \in \N^*$, $\phi : \mathbb{R}^{l}\rightarrow \mathbb{R}$ is $C_{b}^{\infty}$ and  $ t_{1} < . . . <t_{l}$ are some real numbers. The derivative of $F$ with respect to $X$ is the element of $L^{2}(\Omega, \mathscr{H})$ defined by \[ D_{s}F = \sum_{i=1}^{l}\frac{\partial\phi}{\partial x_{i}}(X_{t_{1}}, ...,X_{t_{l}})\xi_{t_{i}}(s), \text{\: \: \:} s \in \R.\]In particular $D_{s}X_{t} = \xi_t(s)$. For any integer $k \geq 1$, we denote by $\mathbb{D}^{k,2}$ the closure of the set of smooth random variables with respect to the norm
   \[ \|F\|_{k,2}^{2} = E(F^{2}) + \sum_{j=1}^{k} E[ \|D^{j}F\|_{\mathscr{H}^{\otimes j}}^{2}].\]
   The Malliavin derivative $D$ satisfies the chain rule. If $\varphi : \mathbb{R}^{n} \rightarrow \mathbb{R}$ is $C_{b}^{1}$ and if $F_1,\ldots,F_n$ are in $\mathbb{D}^{1,2}$, then $\varphi(F_{1},...,F_{n}) \in \mathbb{D}^{1,2}$ and we have
   \[ D\varphi(F_{1},...,F_{n}) = \sum_{i=1}^{n} \frac{\partial \varphi}{\partial x_{i}}(F_{1},...,F_{n})DF_{i}.\]
   We have the following Leibniz formula, whose proof is straightforward  by induction on $q$. Let $\varphi, \psi \in C_{b}^{q}$ $(q\geq 1)$, and fix $0 \leq u<v $ and $0\leq s<t .$ Then $\varphi(X_{t}-X_{s})\psi(X_{v}-X_{u})\in \mathbb{D}^{q,2}$ and
   \begin{eqnarray}
D^{q}\big( \varphi(X_{t}-X_{s})\psi(X_{v}-X_{u})\big) = \sum_{a=0}^{2q} \binom{2q}{a} \varphi^{(a)}(X_{t}-X_{s})\psi^{(2q-a)}(X_{v}-X_{u})\textbf{1}_{[s,t]}^{\otimes a}\tilde{\otimes}\textbf{1}_{[u,v]}^{\otimes (2q-a)},\notag\\
\label{Leibnitz1}
   \end{eqnarray}
   where $\tilde{\otimes}$ stands for the symmetric tensor product. A similar statement holds fo $ u<v \leq 0 $ and $ s<t \leq 0$.

   If a random element $u \in L^{2}(\Omega, \mathscr{H})$ belongs to the domain of the divergence operator, that is, if it satisfies
   \[ |E\langle DF,u\rangle_{\mathscr{H}}|\leq c_{u}\sqrt{E(F^{2})} \text{\: for  any\:} F\in \mathscr{F},\] then $I(u)$ is defined by the duality relationship
\[
E \big( FI(u)\big) = E \big( \langle DF,u\rangle_{\mathscr{H}}\big),
\]
for every $F \in \mathbb{D}^{1,2}.$

For every $n\geq 1$, let $\mathbb{H}_{n}$ be the $n$th Wiener chaos of $X$, that is, the closed linear subspace of $ L^{2}(\Omega, \mathscr{A},P)$ generated by the random variables $\lbrace H_{n}(B(h)), h \in \mathscr{H}, \|h\|_{\mathscr{H}}=1 \rbrace,$ where $H_{n}$ is the $n$th Hermite polynomial. The mapping
$I_{n}(h^{\otimes n}) = H_{n}(B(h))$
 provides a linear isometry between the symmetric tensor product $\mathscr{H}^{\odot n}$ and $\mathbb{H}_{n}$. For $H =\frac{1}{2}$, $I_{n}$ coincides with the multiple Wiener-It\^{o} integral of order $n$. The following duality formula holds
\begin{eqnarray}\label{10}
  E \big( FI_{n}(h)\big) = E \big( \langle D^{n}F,h\rangle_{\mathscr{H}^{\otimes n}}\big),
  \end{eqnarray}
  for any element $ h\in \mathscr{H}^{\odot n}$ and any random variable $F \in \mathbb{D}^{n,2}.$

  Let $\lbrace e_{k}, k \geq 1\rbrace $ be a complete orthonormal system in $\mathscr{H}.$ Given $f \in \mathscr{H}^{\odot n}$ and $g \in \mathscr{H}^{\odot m},$ for every $r= 0,...,n\wedge m,$ the contraction of $f$ and $g$ of order $r$ is the element of $ \mathscr{H}^{\otimes(n+m-2r)}$ defined by
  \[ f\otimes_{r} g = \sum_{k_{1},...,k_{r} =1}^{\infty} \langle f, e_{k_{1}}\otimes...\otimes e_{k_{r}}\rangle_{\mathscr{H}^{\otimes r}}\otimes \langle g,e_{k_{1}}\otimes...\otimes e_{k_{r}}\rangle_{\mathscr{H}^{\otimes r}}.\]
 Note that $f\otimes_{r} g$ is not necessarily symmetric: we denote its symmetrization by $f\tilde{\otimes}_{r} g \in \mathscr{H}^{\odot (n+m-2r)}.$ Finally, we recall the following product formula: if $f\in \mathscr{H}^{\odot n}$ and $g\in \mathscr{H}^{\odot m}$ then
  \begin{eqnarray}\label{11}
  I_{n}(f)I_{m}(g)= \sum_{r=0}^{n\wedge m} r! \binom{n}{r}\binom{m}{r} I_{n+m-2r}(f\tilde{\otimes}_{r}g).
  \end{eqnarray}

 \subsection{Notation and reduction of the problem}
  Throughout all the proof, we shall use the following notation.
For all $k,n \in \N$ we write
\begin{eqnarray*}
\xi_{k2^{-n/2}} &=&\textbf{1}_{[0,k2^{-n/2}]},\quad \xi_{k2^{-n/2}}^-=\textbf{1}_{[-k2^{-n/2},0]},\\
\delta_{k2^{-n/2}} &=& \textbf{1}_{[(k-1)2^{-n/2},k2^{-n/2}]},\quad
\delta_{k2^{-n/2}}^- = \textbf{1}_{[-k2^{-n/2},(-k+1)2^{-n/2}]}.
\end{eqnarray*}
Also, $\langle\cdot,\cdot\rangle$ ($\|\cdot\|$, respectively) will always stand for inner product (the norm, respectively) in an appropriate tensor product $\mathscr{H}^{\otimes s}$.\\

 In the sequel,  we only consider the case $H<\frac12$. The proof of (\ref{Ito1}) in the case $H>\frac12$ is easier and left to the reader,
  whereas the proof when $H=\frac12$ was already
done in \cite{5,6} by Khoshnevisan and Lewis. \\

That said, we now divide the proof of Theorem \ref{main} in several steps.

\subsection{\underline{Step 1}: A key algebraic lemma}
For each
integer $n\geq 1$, $k\in\Z$ and real number $t\geq 0$, let $U_{j,n}(t)$ (resp.
$D_{j,n}(t)$) denote the number of \textit{upcrossings} (resp.
\textit{downcrossings}) of the interval
$[j2^{-n/2},(j+1)2^{-n/2}]$ within the first $\lfloor 2^n
t\rfloor$ steps of the random walk $\{Y(T_{k,n})\}_{k\geq 1}$, that is,
\begin{eqnarray}
U_{j,n}(t)=\sharp\big\{k=0,\ldots,\lfloor 2^nt\rfloor -1 :&&
\notag
\\ Y(T_{k,n})\!\!\!\!&=&\!\!\!\!j2^{-n/2}\mbox{ and }Y(T_{k+1,n})=(j+1)2^{-n/2}
\big\}; \notag\\
D_{j,n}(t)=\sharp\big\{k=0,\ldots,\lfloor 2^nt\rfloor -1:&&
\notag
\\ Y(T_{k,n})\!\!\!\!&=&\!\!\!\!(j+1)2^{-n/2}\mbox{ and }Y(T_{k+1,n})=j2^{-n/2}
\big\}.\notag
\end{eqnarray}
While easy, the following lemma taken from \cite[Lemma 2.4]{5} is going to be the key when
studying the asymptotic behavior of the weighted power variation $V_n^{(2r-1)}(f,t)$ of {\it odd} order $2r-1$, defined as:
\[
V_n^{(2r-1)}(f,t)=\sum_{k=0}^{\lfloor 2^n t \rfloor -1} \frac12\big(f(Z_{T_{k,n}})+f(Z_{T_{k+1,n}})\big)(Z_{T_{k+1,n}}-Z_{T_{k,n}})^{2r-1},\quad t\geq 0.
\]
Its main feature is to separate $X$ from $Y$, thus providing a representation of
$V_n^{(2r-1)}(f,t)$ which is amenable to analysis.

\begin{lemme}\label{lm-kl}
Fix $f\in C^\infty_b$, $t\geq 0$ and $r\in \N^{*}$.
Then
\begin{eqnarray}
&&V_n^{(2r-1)}(f,t)\notag\\
&=& \sum_{j\in\Z}
\frac12\left(
f(X_{j2^{-\frac{n}{2}}}) + f(X_{(j+1)2^{-\frac{n}{2}}})\right) \big(X_{(j+1)2^{-\frac{n}2}}-
X_{j2^{-\frac{n}{2}}}\big)^{2r-1} \notag
\big(U_{j,n}(t)-D_{j,n}(t)\big).
\end{eqnarray}
\end{lemme}
Observe that $V_n^{(1)}(f,t)=V_n(f,t)$, see (\ref{v}).

\subsection{\underline{Step 2}: Transforming the weighted power variations of odd order}

By \cite[Lemma 2.5]{5}, one has
\[
U_{j,n}(t) - D_{j,n}(t)= \left\{
\begin{array}{lcl}
 1_{\{0\leq j< j^*(n,t)\}} & &\mbox{if $j^*(n,t) > 0$}\\
 0 & &\mbox{if $j^{*} = 0$}\\
  -1_{\{j^*(n,t)\leq j<0\}} & &\mbox{if $j^*(n,t)< 0$}
  \end{array}
\right.,
\]
where $j^*(n,t)=2^{n/2}Y_{T_{\lfloor 2^n t\rfloor,n}}$.
As a consequence, $V_n^{(2r-1)}(f,t)$ is equal to
\begin{eqnarray*}
\left\{
\begin{array}{lcl}
2^{-nH(r-\frac12)}\sum_{j=0}^{j^*(n,t)-1}\frac12\big(f(X^+_{j2^{-n/2}}) + f(X^+_{(j+1)2^{-n/2}})\big) \big(X^{n,+}_{j+1}- X^{n,+}_j\big)^{2r-1} & &\mbox{if $j^*(n,t) > 0$}\\
  0 &&\mbox{if $j^{*} = 0$}\\
  2^{-nH(r-\frac12)}\sum_{j=0}^{|j^*(n,t)|-1} \frac12\big(f(X^-_{j2^{-n/2}}) + f(X^-_{(j+1)2^{-n/2}})\big) \big(X^{n,-}_{j+1}- X^{n,-}_j\big)^{2r-1} & &\mbox{if $j^*(n,t) < 0$}
  \end{array}
\right.,
\end{eqnarray*}
where $X^+_t := X_t$ for $t\geq 0$, $X^-_{-t} :=X_t$ for $t<0$, $X^{n,+}(t) := 2^{nH/2}X^+_{2^{-n/2}t}$ for $ t \geq 0$ and $X^{n,-}(-t) := 2^{nH/2}X^-_{2^{-n/2}(-t)}$ for $t < 0$.

Let us now
introduce the following sequence of processes $W_{\pm,n}^{(2r-1)}$, in which $H_p$ stands for the $p$th Hermite polynomial:
\begin{eqnarray}
W_{\pm,n}^{(2r-1)}(f,t)&=&
 \sum_{j=0}^{\lfloor 2^{n/2}t\rfloor -1} \frac12\big(f(X^\pm_{j2^{-\frac{n}{2}}}) + f(X^\pm_{(j+1)2^{-\frac{n}{2}}})\big) H_{2r-1}(X^{n,\pm}_{j+1}- X^{n,\pm}_{j}),\quad t \geq 0 \notag\\
 W_{n}^{(2r-1)}(f,t)&=&\left \{ \begin{array}{lc}
                      W_{+,n}^{(2r-1)}(f,t) &\text{if $t \geq 0$}\notag\\
                      W_{-,n}^{(2r-1)}(f,-t) &\text{if $t < 0$}
                      \end{array}
                      \right. .
\end{eqnarray}
We then have, using the decomposition
$x^{2r-1}=\sum_{l=1}^{r}a_{r,l}H_{2l-1}(x)$ (with $a_{r,r}=1$, which is the only explicit value of $a_{l,r}$ we will need in the sequel),
\begin{eqnarray}
V_n^{(2r-1)}(f,t)
= 2^{-nH(r-\frac12)}\sum_{l=1}^{r}a_{r,l}W_{n}^{(2l-1)}(f,Y_{T_{\lfloor 2^n t\rfloor,n}}).
\label{ref2}
\end{eqnarray}

\subsection{\underline{Step 3}: Known results for fractional Brownian motion}
We recall the following result taken\footnote{More precisely: a careful inspection  would show that there is no additional difficulty to prove (\ref{fdd}) by following the same route than the one used to show \cite[Theorem 1, (1.15)]{8}. The only difference is that the definition
of $W^{(r)}_{\pm,n}$ is of symmetric type, whereas all the quantities of interest studied in \cite{8} are of forward type.}   from \cite{8} .
If $m\geq 2$ and $H\in\big(\frac1{4m-2},\frac12\big)$ then,
for any $f \in C_b^{\infty}$ and as $n\to\infty$,
\begin{eqnarray}
\bigg( X_t, 2^{-n/4}W_{\pm,n}^{(2m-1)}(f,t)  \bigg)_{t \geq 0} \label{fdd}\overset{\rm fdd}{\longrightarrow} \bigg( X_t,\kappa_{2m-1}\int_0^t f(X^{\pm}_s)dW^{\pm}_s \bigg)_{t\geq 0},
\end{eqnarray}
where $W^{+}_t=W_t$ if $t>0$ and
$W^{-}_t=W_{-t}$ if $t<0$, with $W$ a two-sided Brownian motion independent of $X$, and where $\int_0^t f(X^{\pm}_s)dW^{\pm}_s$ must be  understood in the Wiener-It\^o sense.

Note that in the boundary case $m=2$ and $H=\frac16$, (\ref{fdd}) continues
to hold, as was shown in \cite[Theorem 3.1]{11}.

In the case $m=1$, it was shown in
\cite[Theorem 4]{8} (case $H>\frac16$) and
\cite[Theorem 2.13]{11} (case $H=\frac16$)
that, for any fixed $t>0$, the sequence $W_{\pm,n}^{(1)}(f,t)$ converges in probability (when $H>\frac16$) or only in law (when $H=\frac16$)  to a non degenerate limit as $n\to\infty$.

\subsection{\underline{Step 4}: Moment bounds for $W_n^{(2r-1)}(f,\cdot)$}
Fix an integer $r\geq 1$ as well as a function $f\in C^\infty_b$. We claim the existence of $c >0$   such that, for all real numbers $s<t$ and all $n\in\N$,
\begin{eqnarray}
E\big[ \big(W_{n}^{(2r-1)}(f,t) -W_{n}^{(2r-1)}(f,s) \big)^2\big]\label{dernierecomparaison}
&\leq& c\,\max(|s|^{2H},|t|^{2H})\big(|t-s|2^{n/2}+1\big).
\end{eqnarray}
In order to prove  (\ref{dernierecomparaison}), we will  need the following lemma.

\begin{lemma}\label{tech-lem}
If $s,t,u>0$ or if $s,t,u<0$ then
\begin{eqnarray}\label{12}
 |E\big(X_{u}(X_{t}-X_{s})\big)| \leq |t-s|^{2H}.
 \end{eqnarray}

\end{lemma}
\begin{proof}{ }
 When $s,t,u>0$ we have
 \[
 E\big( X_{u}(X_{t}-X_{s})\big) = \frac{1}{2}\big( t^{2H} -s^{2H}\big) + \frac{1}{2}\big(|s-u|^{2H} -|t-u|^{2H}\big).
 \]
Since   $|b^{2H} - a^{2H}|\leq |b-a|^{2H}$ for any $a, b \in \mathbb{R}_{+}$, we immediately deduce (\ref{12}).
The proof when $s,t,u<0$ is similar.

\end{proof}

We are now ready to show (\ref{dernierecomparaison}). We distinguish two cases according to the signs of $s,t \in \R$ (and reducing the problem by symmetry):

\begin{itemize}
\item[(1)] if  $0 \leq s < t$ (the case $s<t\leq 0$ being similar), then
\end{itemize}
\begin{eqnarray*}
 && E[(W_{n}^{(2r-1)}(f,t)- W_{n}^{(2r-1)}(f,s))^2]= E[(W_{+,n}^{(2r-1)}(f,t)-W_{+,n}^{(2r-1)}(f,s))^2]\\
&=& \frac14 \sum_{j,j'=\lfloor 2^{n/2}s \rfloor}^{\lfloor 2^{n/2}t\rfloor -1} \bigg|E \big[ \big(f(X^+_{j2^{-\frac{n}{2}}}) + f(X^+_{(j+1)2^{-\frac{n}{2}}})\big)\\
&& \times \big(f(X^+_{j'2^{-\frac{n}{2}}}) + f(X^+_{(j'+1)2^{-\frac{n}{2}}})\big)H_{2r-1}(X^{n,+}_{j+1}-X^{n,+}_{j})H_{2r-1}(X^{n,+}_{j'+1}-X^{n,+}_{j'}) \big]\bigg|\\
&=& \frac14 2^{nH(2r-1)}\sum_{j,j'=\lfloor 2^{n/2}s \rfloor}^{\lfloor 2^{n/2}t\rfloor -1} \bigg|E \big[
\Theta_j^n f(X^{+})\Theta_{j'}^n f(X^{+}) I_{2r-1}(\delta_{(j+1)2^{-n/2}}^{\otimes (2r-1)})I_{2r-1}(\delta_{(j'+1)2^{-n/2}}^{\otimes (2r-1)})\big]\bigg|,
\end{eqnarray*}
with obvious notation. Relying to the product formula (\ref{11}), we deduce that this latter quantity is less than or equal to
\begin{eqnarray}
&& \frac14 2^{nH(2r-1)}\sum_{j,j'= \lfloor 2^{n/2}s \rfloor}^{\lfloor 2^{n/2}t \rfloor -1}\sum_{l=0}^{2r-1}l!\binom{2r-1}{l}^2 \big|\langle \delta_{(j+1)2^{-n/2}}; \delta_{(j'+1)2^{-n/2}}\rangle \big|^l \notag\\
&& \hskip2.2cm \times \bigg| E\big[ \Theta_j^n f(X^+)\Theta_{j'}^n f(X^+)I_{4r-2-2l}(\delta_{(j+1)2^{-n/2}}^{\otimes (2r-1-l)}\tilde{\otimes}\delta_{(j'+1)2^{-n/2}}^{\otimes (2r-1-l)})
\big] \bigg|\notag\\
&=:& \frac14\sum_{l=0}^{2r-1} l! \binom{2r-1}{l}^2 Q_n^{(+,r,l)}(s,t). \label{T1}
\end{eqnarray}
By the duality formula (\ref{10}) and the Leibniz rule (\ref{Leibnitz1}), one has that
\begin{eqnarray*}
&& d_n^{(+ ,r,l)}(j,j'):=E\big[ \Theta_j^n f(X^+)\Theta_{j'}^n f(X^+)I_{4r-2-2l}(\delta_{(j+1)2^{-n/2}}^{\otimes (2r-1-l)}\tilde{\otimes}\delta_{(j'+1)2^{-n/2}}^{\otimes (2r-1-l)})
\big]\\
&\vspace{0.5cm}&\\
&=& E \big[ \big \langle D^{4r-2-2l}(\Theta_j^n f(X^+)\Theta_{j'}^n f(X^+))\: ; \: \delta_{(j+1)2^{-n/2}}^{\otimes (2r-1-l)}\tilde{\otimes}\delta_{(j'+1)2^{-n/2}}^{\otimes (2r-1-l)} \big \rangle \big]\\
&=& \sum_{a=0}^{4r-2-2l}\binom{4r-2-2l}{a} E\bigg[ \bigg \langle \big(f^{(a)}(X^+_{j2^{-n/2}})\xi_{j2^{-n/2}}^{\otimes a} + f^{(a)}(X^+_{(j+1)2^{-n/2}})\xi_{(j+1)2^{-n/2}}^{\otimes a} \big)\\
&& \tilde{\otimes} \big(f^{(4r-2-2l-a)}(X^+_{j'2^{-n/2}})\xi_{j'2^{-n/2}}^{\otimes (4r-2-2l-a)} + f^{(4r-2-2l-a)}(X^+_{(j'+1)2^{-n/2}})\xi_{(j'+1)2^{-n/2}}^{\otimes (4r-2-2l-a)} \big) \: ; \: \\
&& \hskip9.6cm \delta_{(j+1)2^{-n/2}}^{\otimes (2r-1-l)}\tilde{\otimes}\delta_{(j'+1)2^{-n/2}}^{\otimes (2r-1-l)} \bigg \rangle \bigg].
\end{eqnarray*}
Let now $c$ denote a generic constant that may differ from one line to another and recall that $f \in C_{b}^{\infty}$. We then have the following estimates.

\begin{itemize}
\item \underline{Case $l=2r-1$}
\end{itemize}
\begin{eqnarray}
&&Q_n^{(+,r,2r-1)}(s,t)\notag\\
 &\leq & c\,2^{nH(2r-1)}\sum_{j,j'= \lfloor 2^{n/2}s\rfloor }^{\lfloor 2^{n/2}t \rfloor -1} \big|\langle \delta_{(j+1)2^{-n/2}}; \delta_{(j'+1)2^{-n/2}}\rangle \big|^{2r-1} \notag \\
&=& c \sum_{j,j'= \lfloor 2^{n/2}s\rfloor}^{\lfloor 2^{n/2}t \rfloor -1}
\big|\frac12( |j-j'+1|^{2H} + |j-j'-1|^{2H} - 2 |j-j'|^{2H}) \big|^{2r-1}\notag \\
&=&  c \sum_{j = \lfloor 2^{n/2}s\rfloor }^{\lfloor 2^{n/2}t \rfloor -1}\sum_{q =j - \lfloor 2^{n/2}t \rfloor + 1}^{j - \lfloor 2^{n/2}s\rfloor } \big|\rho(q) \big|^{2r-1},\notag
\end{eqnarray}
with $\rho(q):= \frac12( |q+1|^{2H} + |q-1|^{2H} - 2 |q|^{2H})$. By a Fubini argument, it comes
\begin{eqnarray}
&&Q_n^{(+,r,2r-1)}(s,t)\notag\\
&\leq& c \!\!\!\sum_{q = \lfloor 2^{n/2}s \rfloor  - \lfloor 2^{n/2}t \rfloor +1}^{\lfloor 2^{n/2}t \rfloor - \lfloor 2^{n/2}s \rfloor  -1} \!\!\!\!\!|\rho(q)|^{2r-1} \bigg( (q+ \lfloor 2^{n/2}t \rfloor)\wedge
\lfloor 2^{n/2}t \rfloor - ( q+ \lfloor 2^{n/2}s\rfloor )\vee \lfloor 2^{n/2}s\rfloor \bigg)\notag\\
&\leq&  c  \sum_{q = \lfloor 2^{n/2}s \rfloor  - \lfloor 2^{n/2}t \rfloor +1}^{\lfloor 2^{n/2}t \rfloor - \lfloor 2^{n/2}s \rfloor  -1} |\rho(q)|^{2r-1}\big( \lfloor 2^{n/2}t \rfloor - \lfloor 2^{n/2}s \rfloor \big)\notag\\
&\leq&c \sum_{q \in \Z}|\rho(q)|^{2r-1}\big|\lfloor 2^{n/2}t \rfloor - \lfloor 2^{n/2}s \rfloor \big|=c\big|\lfloor 2^{n/2}t \rfloor - \lfloor 2^{n/2}s \rfloor \big|\notag\\
& \leq& c\big(\big|\lfloor 2^{n/2}t \rfloor -2^{n/2}t \big| + 2^{n/2}\big|t-s\big| +  \big|\lfloor 2^{n/2}s \rfloor -2^{n/2}s\big|\big) \notag\\
&\leq& c ( 1 + 2^{n/2}|t-s|). \label{T2}
\end{eqnarray}
Note that $\sum_{q \in \Z}|\rho(q)|^{2r-1} < \infty$ since $H<\frac12\leq 1-\frac{1}{4r-2}$.

\begin{itemize}
\item \underline{Preparation to the cases where $ 0 \leq l \leq 2r-2$}
\end{itemize}

In order to handle the terms $Q_n^{(+,r,l)}(s,t)$  whenever $ 0 \leq l \leq 2r-2$, we will make use of the following decomposition:
\begin{equation}
|d_n^{(+, r,l)}(j,j')| \leq \sum_{u,v=0}^1  \Omega^{(u,v,r,l)}_n(j,j'),\label{T3}
\end{equation}
where
\begin{eqnarray*}
 \Omega^{(u,v,r,l)}_n(j,j') &=& \sum_{a=0}^{4r-2-2l}\binom{4r-2-2l}{a} \big| E[f^{(a)}(X^+_{(j+u)2^{-n/2}})f^{(4r-2-2l-a)}(X^+_{(j'+v)2^{-n/2}})]\big| \\
 &&\times\big| \big \langle \xi_{(j+u)2^{-n/2}}^{\otimes a}\tilde{\otimes}\xi_{(j'+v)2^{-n/2}}^{\otimes (4r-2-2l-a)}; \delta_{(j+1)2^{-n/2}}^{\otimes (2r-1-l)}\tilde{\otimes}\delta_{(j'+1)2^{-n/2}}^{\otimes (2r-1-l)}\big \rangle \big|.
 \end{eqnarray*}

\begin{itemize}
\item \underline{Case $ 1\leq l \leq 2r-2$} (only when $r\geq 2$)
\end{itemize}

 Since $f$ belongs to $C_b^{\infty}$ and since, by (\ref{12}), we have
 $|\langle \xi_t ; \delta_{(j+1)2^{-n/2}} \rangle| \leq 2^{-nH}$ for all $ t \geq 0$ and all $ j \in \N$,  we deduce that
 \[
 |d_n^{(+ ,r,l)}(j,j')| \leq c\,2^{-nH(4r-2-2l)}.
 \]
 As a consequence, and relying to the same arguments that have been used previously in the case $l =2r-1$, we get
 \begin{eqnarray}
 Q_n^{(+,r,l)}(s,t)
 &\leq& c \,2^{-nH(4r-2-2l)}\: 2^{nH(2r-1)}\sum_{j,j'= \lfloor 2^{n/2}s\rfloor }^{\lfloor 2^{n/2}t \rfloor -1} \big|\langle \delta_{(j+1)2^{-n/2}}; \delta_{(j'+1)2^{-n/2}}\rangle \big|^l \notag\\
 & \leq& c\, 2^{-nH(4r-2-2l)}\: 2^{nH(2r-1)} 2^{-nHl} \sum_{q \in \Z}|\rho(q)|^l ( 1 + 2^{n/2}|t-s|)\notag\\
 & =& c\: 2^{-nH(2r-1-l)}(1 + 2^{n/2}|t-s|)\leq c\: (1 + 2^{n/2}|t-s|). \label{T4}
 \end{eqnarray}

\begin{itemize}
\item \underline{Case $ l=0$}
\end{itemize}
Relying to the decomposition (\ref{T3}), we get
 \begin{eqnarray}
   Q_n^{(+,r,0)}(s,t) &\leq&  2^{nH(2r-1)}\sum_{j,j'= \lfloor 2^{n/2}s\rfloor }^{\lfloor 2^{n/2}t \rfloor -1}
\sum_{u,v=0}^1
   \Omega^{(u,v,r,0)}_n(j,j') \label{omega2}.
 \end{eqnarray}

 We will study only the term corresponding to
 $\Omega^{(0,1,r,0)}_n(j,j') $  in (\ref{omega2}), which is representative of the difficulty.
  It is given by
 \begin{eqnarray}
 &&2^{nH(2r-1)}\sum_{j,j'= \lfloor 2^{n/2}s\rfloor}^{\lfloor 2^{n/2}t \rfloor -1} \sum_{a=0}^{4r-2}\binom{4r-2}{a} \big| E[f^{(a)}(X^+_{j2^{-n/2}})f^{(4r-2-a)}(X^+_{(j'+1)2^{-n/2}})]\big|\notag\\
 && \hskip6cm \times \big| \big \langle \xi_{j2^{-n/2}}^{\otimes a}\tilde{\otimes}\xi_{(j'+1)2^{-n/2}}^{\otimes (4r-2-a)}; \delta_{(j+1)2^{-n/2}}^{\otimes (2r-1)}\tilde{\otimes}\delta_{(j'+1)2^{-n/2}}^{\otimes (2r-1)}\big \rangle \big|\notag\\
 &\leq &  c\, 2^{nH(2r-1)}\sum_{j,j'= \lfloor 2^{n/2}s\rfloor }^{\lfloor 2^{n/2}t \rfloor -1} \sum_{a=0}^{4r-2}\big| \big \langle \xi_{j2^{-n/2}}^{\otimes a}\tilde{\otimes}\xi_{(j'+1)2^{-n/2}}^{\otimes (4r-2-a)};  \notag
  \delta_{(j+1)2^{-n/2}}^{\otimes (2r-1)}\tilde{\otimes}\delta_{(j'+1)2^{-n/2}}^{\otimes (2r-1)}\big \rangle \big|.
 \end{eqnarray}
We define $E_n^{(a,r)}(j,j'):= \big| \big \langle \xi_{j2^{-n/2}}^{\otimes a}\tilde{\otimes}\xi_{(j'+1)2^{-n/2}}^{\otimes (4r-2-a)}; \delta_{(j+1)2^{-n/2}}^{\otimes (2r-1)}\tilde{\otimes}\delta_{(j'+1)2^{-n/2}}^{\otimes (2r-1)}\big \rangle \big|.$
By (\ref{12}), recall that $|\langle \xi_t ; \delta_{(j+1)2^{-n/2}} \rangle| \leq 2^{-nH}$ for all $ t \geq 0$ and all $ j \in \N$. We thus get, with $\tilde{c}_a$  some combinatorial constants,
  \begin{eqnarray*}
  E_n^{(a,r)}(j,j')  &\leq& \tilde{c}_a\,2^{-nH(4r-3)}
  \big(
   |\langle \xi_{j2^{-n/2}}; \delta_{(j+1)2^{-n/2}}\rangle|
  + | \langle \xi_{j2^{-n/2}}; \delta_{(j'+1)2^{-n/2}}\rangle|\\
&&  \hskip2.5cm+  |\langle \xi_{(j'+1)2^{-n/2}}; \delta_{(j+1)2^{-n/2}}\rangle|
  + | \langle \xi_{(j'+1)2^{-n/2}}; \delta_{(j'+1)2^{-n/2}}\rangle|
  \big).
  \end{eqnarray*}
For instance,
  we can write
  \begin{eqnarray*}
 && \sum_{j,j'=\lfloor 2^{n/2}s\rfloor}^{\lfloor 2^{n/2}t\rfloor -1}
 | \langle \xi_{(j'+1)2^{-n/2}}; \delta_{(j+1)2^{-n/2}}\rangle|\\
&  =&2^{-nH-1} \sum_{j,j'=\lfloor 2^{n/2}s\rfloor}^{\lfloor 2^{n/2}t\rfloor -1}
 \big|
 (j+1)^{2H}-j^{2H}+|j'-j+1|^{2H}-|j'-j|^{2H}
 \big|\\
 &\leq& 2^{-nH-1} \sum_{j,j'=\lfloor 2^{n/2}s\rfloor}^{\lfloor 2^{n/2}t\rfloor -1}
 \big((j+1)^{2H}-j^{2H}\big)\\
 &&+
 2^{-nH-1} \sum_{\lfloor 2^{n/2}s\rfloor\leq j\leq j'\leq \lfloor 2^{n/2}t\rfloor-1}
 \big((j'-j+1)^{2H}-(j'-j)^{2H}\big)\\
  &&+
 2^{-nH-1} \sum_{\lfloor 2^{n/2}s\rfloor\leq j'<j\leq \lfloor 2^{n/2}t\rfloor-1}
 \big((j-j')^{2H}-(j-j'-1)^{2H}\big)\\
 &\leq&\frac32\,2^{-nH}\big(\lfloor 2^{n/2}t\rfloor - \lfloor 2^{n/2}s\rfloor\big) \lfloor 2^{n/2}t\rfloor ^{2H} \leq \frac{3t^{2H}}2\,\big(1+2^{n/2}|t-s|\big).
  \end{eqnarray*}
  Similarly,
  \begin{eqnarray*}
   \sum_{j,j'=\lfloor 2^{n/2}s\rfloor}^{\lfloor 2^{n/2}t\rfloor -1}
  \langle \xi_{j2^{-n/2}}; \delta_{(j+1)2^{-n/2}}\rangle|&\leq&\frac{3t^{2H}}2\,\big(1+2^{n/2}|t-s|\big);\\
   \sum_{j,j'=\lfloor 2^{n/2}s\rfloor}^{\lfloor 2^{n/2}t\rfloor -1}
  \langle \xi_{j2^{-n/2}}; \delta_{(j'+1)2^{-n/2}}\rangle|&\leq&\frac{3t^{2H}}2\,\big(1+2^{n/2}|t-s|\big);\\
     \sum_{j,j'=\lfloor 2^{n/2}s\rfloor}^{\lfloor 2^{n/2}t\rfloor -1}
  \langle \xi_{(j'+1)2^{-n/2}}; \delta_{(j'+1)2^{-n/2}}\rangle|&\leq&\frac{3t^{2H}}2\,\big(1+2^{n/2}|t-s|\big).
  \end{eqnarray*}
  As a consequence, we deduce
  \begin{equation}
  Q_n^{(+,r,0)}(s,t)\leq  c\,2^{-nH(2r-2)} t^{2H}\big(2^{n/2}|t-s|+1)\leq  c\,t^{2H}\big(2^{n/2}|t-s|+1). \label{T5}
  \end{equation}

Combining (\ref{T1}), (\ref{T2}), (\ref{T4}) and (\ref{T5}) finally shows our claim (\ref{dernierecomparaison}).

\begin{itemize}

\item[(2)] if $s < 0 \leq t$, then
\end{itemize}
 \begin{eqnarray*}
&&E[(W_{n}^{(2r-1)}(f,t)- W_{n}^{(2r-1)}(f,s))^2]= E[(W_{+,n}^{(2r-1)}(f,t)- W_{-,n}^{(2r-1)}(f,-s))^2]\\
& \leq & 2E[(W_{+,n}^{(2r-1)}(f,t))^2] + 2E[(W_{-,n}^{(2r-1)}(f,-s))^2].
\end{eqnarray*}
By (1) with $s=0$, one can write
 \begin{eqnarray*}
 && E[(W_{+,n}^{(2r-1)}(f,t))^2]\leq c\,t^{2H}(t2^{n/2}+1).
\end{eqnarray*}
Similarly
\begin{eqnarray*}
 && E[(W_{-,n}^{(2r-1)}(f,-s))^2]\leq c(-s)^{2H}\big((-s)2^{n/2}+1\big)
\end{eqnarray*}
We deduce that
\begin{eqnarray*}
&&E[(W_{n}^{(2r-1)}(f,t)- W_{n}^{(2r-1)}(f,s))^2] \leq c \max(t^{2H},(-s)^{2H})\big((t-s)2^{n/2}+1\big).
\end{eqnarray*}
That is, (\ref{dernierecomparaison}) also holds true in this case.

\subsection{\underline{Step 5}: Limits of the weighted power variations of odd order}
Fix $f \in C_b^{\infty}$ and $t\geq 0$. We claim that, if $H\in\big[\frac16,\frac12\big)$ and $r\geq 3$ then, as $n\to\infty$,
\begin{equation}\label{proposition2}
V_n^{(2r-1)}(f,t) \overset{\rm prob}{\longrightarrow}  0.
\end{equation}
Moreover, if $H\in\big(\frac16,\frac12\big)$ then, as $n\to\infty$,
\begin{equation}\label{proposition2bis}
V_n^{(3)}(f,t)  \overset{\rm prob}{\longrightarrow}  0,
\end{equation}
whereas, if $H=\frac16$ then, as $n\to\infty$,
\begin{eqnarray}
\left(
X_t,Y_t,V_n^{(3)}(f,t)
\right)_{t\geq 0}\overset{\rm fdd}{\to}
\left(
X_t,Y_t,\kappa_3\int_0^{Y_t}f(X_s)dW_s
\right)_{t\geq 0},\label{cubic}
\end{eqnarray}
with $W=(W_t)_{t\in \R}$ a two-sided Brownian motion independent of the pair $(X,Y)$.

Indeed, using the decomposition (\ref{ref2}), the conclusion of Step 4 (to pass from $Y_{T_{\lfloor 2^n t\rfloor,n}}$ to $Y_t$)  and the convergence (\ref{lemma2.2}), we deduce that the limit
of
$V_n^{(2r-1)}(f,t) $
is the same as that of
\[
2^{-nH(r-\frac12)}\sum_{l=1}^{r}a_{r,l}W_{n}^{(2l-1)}(f,Y_t).
\]

Thus, the proofs of (\ref{proposition2}), (\ref{proposition2bis}) and (\ref{cubic}) then follow directly from  the results recalled in Step 3, as well as the fact that $X$ and $Y$ are independent.

\subsection{\underline{Step 6}: Proving (\ref{Ito1}) and (\ref{Ito2})}

We assume $H\in[\frac16,\frac12)$.
We will make use of the following Taylor's type formula.  Fix $f \in C_b^{\infty}$. For any $a,\: b \in \mathbb{R}$
and for some constants $c_r$ whose explicit values are immaterial here,
  \begin{eqnarray}
f(b) - f(a) \notag
 &=& \frac12\big(f'(a)+f'(b)\big)(b-a) -
 \frac1{24}\big(f'''(a)+f'''(b)\big)(b-a)^3 \\
&&+\sum_{r=3}^{7} c_{r}
 \big(f^{(2r-1)}(a)+f^{(2r-1)}(b)\big)(b-a)^{2r-1}
 + O ( |b-a|^{14}),\notag
  \end{eqnarray}
 where $|O ( |b-a|^{14})| \leq C_{f}|b-a|^{14}$, $C_f$ being a constant depending only on $f$.
One can thus write
\begin{eqnarray}
f(Z_{T_{\lfloor 2^n t\rfloor,n}}) - f(0)&=&\sum_{k=0}^{\lfloor 2^{n}t\rfloor -1}\!\!\!\big( f(Z_{T_{k+1,n}}) - f(Z_{T_{k,n}})\big)\notag\\
&\!\!\!=&
\sum_{k=0}^{\lfloor 2^n t\rfloor-1}\frac12\big(f(Z_{T_{k,n}})+f(Z_{T_{k+1,n}})\big) (Z_{T_{k+1,n}}- Z_{T_{k,n}})\label{pm3}\\
&&-\frac1{12}
V_n^{(3)}(f,t)+
\sum_{r=3}^{7}2c_rV_n^{(2r-1)}(f,t)+ \sum_{k=0}^{\lfloor 2^n t \rfloor -1}O ( (Z_{T_{k+1,n}}-Z_{T_{k,n}})^{14})\notag.
\end{eqnarray}

As far as the big $O$ in (\ref{pm3}) is concerned, we have, with $G\sim N(0,1)$,
 \begin{eqnarray}
&& E\big[ \big|\sum_{k=0}^{\lfloor 2^n t \rfloor -1}O ( (Z_{T_{k+1,n}}-Z_{T_{k,n}})^{14})\big|\big] \leq C_f \sum_{k=0}^{\lfloor 2^n t \rfloor -1}E\big[(Z_{T_{k+1,n}}-Z_{T_{k,n}})^{14}\big]\notag\\
& =& C_f \sum_{k=0}^{\lfloor 2^n t \rfloor -1} 2^{-7nH}E[G^{14}] \leq C_f E[G^{14}] t\,2^{n(1-7H)} \to_{n\to\infty} 0\quad\mbox{since $H\geq \frac16$}.\label{pm2}
\end{eqnarray}

On the other hand, by continuity of $f\circ Z$  and due to (\ref{lemma2.2}), one has, almost surely and
as $n\to\infty$,
\begin{equation}\label{pm1}
 f(Z_{T_{\lfloor 2^n t\rfloor,n}}) - f(0)
\to f(Z_t)-f(0).
\end{equation}

Finally, when $H>\frac16$ the desired conclusion (\ref{Ito1}) follows from
(\ref{pm2}), (\ref{pm1}), (\ref{proposition2}) and (\ref{proposition2bis}) plugged into (\ref{pm3}).
The proof of (\ref{Ito2}) when $H=\frac16$ is similar, the only difference being that one has (\ref{cubic}) instead of (\ref{proposition2bis}), thus leading
to the bracket term $\frac{\kappa_3}{12}\int_0^{Y_{t}}f'''(X_s)dW_s$ in (\ref{Ito2}).

\subsection{\underline{Step 7}: Proving (\ref{Ito3})}

Using $b^3-a^3=\frac32(a^2+b^2)(b-a)-\frac12(b-a)^3$, one can write, with ${\bf 1}$ denoting the function constantly equal to 1,
\begin{eqnarray*}
V_n(x\mapsto x^2,t)-\frac13 Z_t^3
&=&\frac16 V_n^{(3)}({\bf 1},t)+\frac13\sum_{k=0}^{\lfloor 2^n t\rfloor -1}
(Z_{T_{k+1,n}}^3- Z_{T_{k,n}}^3)
-\frac13 Z_t^3\\
&=&\frac16 V_n^{(3)}({\bf 1},t)+\frac13
(Z_{T_{\lfloor 2^n t\rfloor ,n}}^3- Z_t^3).
\end{eqnarray*}
As a result, and thank to (\ref{lemma2.2}), one deduces that
if $V_n(x\mapsto x^2,t)$ converges stably in law, then
$V_n^{(3)}({\bf 1},t)$ must converge as well.
But
it is shown in \cite[Corollary 1.2]{13} that
$2^{-n(1-6H)/4}V_n^{(3)}({\bf 1},t)$ converges in law
to a non degenerate limit.
This being clearly in contradiction with the convergence of $V_n^{(3)}({\bf 1},t)$, we deduce that (\ref{Ito3}) holds.

\end{document}